\documentclass{article}
\usepackage{arxiv}
\usepackage{amsmath}
\usepackage{amssymb}
\usepackage{amsthm}
\usepackage{graphicx}
\theoremstyle{plain}

\newtheorem{theorem}{Theorem}[section]
\newtheorem{proposition}{Proposition}[section]
\newtheorem{lemma}{Lemma}[section]
\newtheorem{corollary}{Corollary}[section]
\theoremstyle{definition}
\newtheorem{definition}{Definition}[section]
\newtheorem{remark}{Remark}[section]

\begin{document}
	
	\title{Additive transform of an arithmetic function : Part I}
	
	\author{
		{ \sc Es-said En-naoui } \\ 
		University Sultan Moulay Slimane\\ Morocco\\
		essaidennaoui1@gmail.com\\
		\\
	}
	
	\maketitle

	\begin{abstract}
		The Additive Transform of an arithmetic function represents a novel approach to examining the interplay between multiplicative arithmetic function and additive functions. This transform concept introduces a method to systematically generate new arithmetic functions by combining the values of an existing function under an additive operation. The resulting framework not only extends our understanding of classical arithmetic functions but also provides a versatile tool for exploring additive relationships within the realm of number theory. In this article, we present the fundamental principles of the Additive Transform and illustrate its application through various examples, shedding light on its potential implications for diverse mathematical domains.
		\\
		For all positive integer $n$. a motivation for the present study is to give a new concept named the Additive transform of an arithmetic function $f$ when $f$ equals some special arithmetic functions, that new concept can help us to prove many results  like  : \begin{equation}
			\big(\mu*f.Id\big)(n)
			=\varphi(n)f(n)+\varphi(n)
			\sum \limits_{p^{\alpha}||n} \frac{f(p^{\alpha})-f(p^{\alpha-1})}{p-1}
		\end{equation}
		where $f$ is an additive function .
	\end{abstract}
	
	\section{Introduction}

	First of all, to cultivate analytic number theory one must acquire a considerable 
	skill for operating with arithmetic functions. We begin with a few elementary  considerations. 
	
	\begin{definition}[arithmetic function]
		An \textbf{arithmetic function} is a function $f:\mathbb{N}\longrightarrow \mathbb{C}$ with
		domain of definition the set of natural numbers $\mathbb{N}$ and range a subset of the set of complex numbers $\mathbb{C}$.
	\end{definition}
	
	\begin{definition}[multiplicative function]
		A function $f$ is called an \textbf{multiplicative function} if and
		only if : 
		\begin{equation}\label{eq:1}
			f(nm)=f(n)f(m)
		\end{equation}
		for every pair of coprime integers $n$,$m$. In case (\ref{eq:1}) is satisfied for every pair of integers $n$ and $m$ , which are not necessarily coprime, then the function $f$ is
		called \textbf{completely multiplicative}.
	\end{definition}
	
	Clearly , if $f$ are a multicative function , then
	$f(n)=f(p_1^{\alpha _1})\ldots f(p_s^{\alpha _s})$, 
	for any positive integer $n$ such that 
	$n = p_1^{\alpha _1}\ldots  p_s^{\alpha _s}$ , and if $f$ is completely multiplicative , so we have : $f(n)=f(p_1)^{\alpha _1}\ldots f(p_s)^{\alpha _s}$.

	The functions defined above are widely studied in the literature, see, e.g., \cite{Ap,KM,LT,Mc,SC,Sc, Sh}. 
	
	\bigskip
	\begin{definition}[additive function]
		A function $f$ is called an \textbf{additive function} if and
		only if : 
		\begin{equation}
			f(nm)=f(n)+f(m)
		\end{equation}
		for every pair of coprime integers $n$,$m$. In case (3) is satisfied for every pair of integers $n$ and $m$ , which are not necessarily coprime, then the function $f$ is
		called \textbf{completely additive}.
	\end{definition}
	Clearly , if $f$ are a additive function , then
	$f(n)=f(p_1^{\alpha _1})+\ldots +f(p_s^{\alpha _s})$, 
	for any positive integer $n$ such that 
	$n = p_1^{\alpha _1}\ldots  p_s^{\alpha _s}$ , and if $f$ is completely additive , so we have : $f(n)=\alpha _1f(p_1)+\ldots +\alpha _sf(p_s)$.
	\bigskip
	\begin{definition}[L-additive function]
		We say that an arithmetic function $f$ is {\em Leibniz-additive} (or, {\em L-additive}, in short) if there is a completely multiplicative function $h_f$ such that 
		\begin{equation}\label{gca}
			f(mn)=f(m)h_f(n)+f(n)h_f(m)
		\end{equation}
		for all positive integers $m$ and $n$. 
	\end{definition}
	\medskip
	\section{Additive transform of an arithmetic function .}
	Let us begin by defining the new concept related to thearithmetic function, this concept called \textbf{Additive transform} of arithmetic function.
	\begin{definition}\label{def-2-1}
		Let $f$ an function arithmetic.\\
		The Additive transform of an arithmetic function $f$ is an application  
		$\Phi_f :\mathbb{N} \longrightarrow \mathbb{C}$  verify that condition : 
		\begin{enumerate}
			\item 
			$\forall\;p\;\in\mathbb{P}\;,\;\forall\;\alpha\;\in\mathbb{N}^*\;\;,\;\;
			\Phi_f(p^{\alpha})=f(p^{\alpha}).$
			\item
			$\Phi_f(nm)=n\Phi_f(m)+m\Phi_f(n)$ for relatively prime $m$ and $n$. (Leibniz formula).
		\end{enumerate}
	\end{definition}
	
	\begin{definition}
		Let $f$ an arithmetic function.\\
		The Additive transform of an arithmetic function $f$ is a function $\Phi_f:\mathbb{N}\longrightarrow \mathbb{C}.$,It is :
		\begin{enumerate}
			\item
			\textbf{Additive transform} if $\Phi_f(nm)=n\Phi_f(m)+m\Phi_f(n)$ for relatively prime $m$ and $n$.
			\item
			\textbf{Completely additive transform} if $\Phi_f(nm)=n\Phi_f(m)+m\Phi_f(n)$ for any $m$ and $n$.
		\end{enumerate}
	\end{definition}
	\begin{remark}
		Let $f$ an arithmetic function , if the additive transform of $f$ is completely additive transform , then $\Phi_f$ is L-additive with $h_{\Phi_f}(n)=n$ .
	\end{remark}
	\begin{theorem}\label{the-2-1}
		Let $n$ a positive integer such that $n = p_1^{\alpha _1}\ldots  p_s^{\alpha _s}$ and let $f$ an arithmetic function , So  we have : 
		\begin{equation}
			\Phi_{f}(n)=n\sum \limits_{i=1}^s\frac{f\big(p_i^{\alpha _i}\big)}{p_i^{\alpha _i}}.
		\end{equation}
	\end{theorem}
	
	\begin{proof}
		Let $n$ a positive integer such that $n = p_1^{\alpha _1}\ldots  p_s^{\alpha _s}$ and let $f$ an arithmetic function ,bu definition $(\ref{def-2-1})$ Then :
		\begin{align*}
			\Phi_{f}(n) &=\Phi_{f}(p_1^{\alpha _1}\ldots  p_s^{\alpha _s})
			=p_1^{\alpha_1}\Phi_{f}(p_2^{\alpha _2}\ldots  p_s^{\alpha _s})+p_2^{\alpha _2}\ldots  p_s^{\alpha _s}\Phi_{f}(p_1^{\alpha_1})\\
			& =n\frac{f(p_1^{\alpha_1})}{p_1^{\alpha_1}}
			+p_1^{\alpha_1}\bigg(p_2^{\alpha_2}\Phi_{f}(p_3^{\alpha _3}\ldots  p_s^{\alpha _s})+p_3^{\alpha _3}\ldots  p_s^{\alpha _s}\Phi_{f}(p_2^{\alpha _2})
			\bigg)\\
			&=n\frac{f(p_1^{\alpha_1})}{p_1^{\alpha_1}}+
			n\frac{f(p_2^{\alpha_2})}{p_2^{\alpha_2}}+p_1^{\alpha_1}p_2^{\alpha_2}\Phi_{f}(p_3^{\alpha _3}\ldots  p_s^{\alpha _s})\\
			\vdots\\
			&=n\frac{f(p_1^{\alpha_1})}{p_1^{\alpha_1}}+
			n\frac{f(p_2^{\alpha_2})}{p_2^{\alpha_2}}+\ldots+
			n\frac{f(p_s^{\alpha_s})}{p_s^{\alpha_s}}\\
			&=n\bigg(\frac{f(p_1^{\alpha_1})}{p_1^{\alpha_1}}+\ldots+ \frac{f(p_s^{\alpha_s})}{p_s^{\alpha_s}}\bigg)\\
			&=n\sum \limits_{i=1}^s\frac{f\big(p_i^{\alpha _i}\big)}{p_i^{\alpha _i}}
		\end{align*}
	\end{proof}
	
	\begin{definition}
		Let $n$ a positive integer , and $p\in\;\mathbb{P}$ such that $\alpha=v_p(n)\neq 0$ , So one can define the arithmetic partial derivative of a function arithmetic $f$ via :
		\begin{equation}
			\frac{\partial f}{\partial p}(n)=\frac{f(p^{\alpha})}{p^{\alpha}}n
		\end{equation}
	\end{definition}
	
	Then we have this result :
	\begin{equation}
		\Phi_f(n)=\sum \limits_{p^{\alpha}||n}\frac{\partial f}{\partial p}(n)
	\end{equation}
	
	\begin{proposition}
		Let $n$ a positive integer such that  and let $f$ an arithmetic function multiplicative , So  we have : 
		\begin{equation}
			\Phi_{f}(n)\geq n\omega(n)\bigg(\frac{f(n)}{n}\bigg)^{\frac{1}{\omega(n)}}.
		\end{equation}
	\end{proposition}
	
	\begin{proof}
		Let $n$ a positive integer such that $n = p_1^{\alpha _1}\ldots  p_s^{\alpha _s}$ and let $f$ an arithmetic function .\\
		for every $x_1,x_2,\ldots,x_s\in\mathbb{R}^{+*}$, by The AM-GM inequality we have :
		$$
		\frac{x_1 + x_2 + \ldots + x_n}{n} \geq \sqrt[n]{x_1 \cdot x_2 \cdot \ldots \cdot x_n} 
		$$
		So if we put $x_i=\frac{f(p_i^{\alpha_i})}{p_i^{\alpha_i}}$ we find that: 
		\begin{align*}
			\frac{1}{s}\sum \limits_{i=1}^s \frac{f(p_i^{\alpha_i})}{p_i^{\alpha_i}} \geq 
			\sqrt[s]{\prod \limits_{i=1}^s \frac{f(p_i^{\alpha_i})}{p_i^{\alpha_i}}}
			\Longrightarrow 
			n\sum \limits_{i=1}^s \frac{f(p_i^{\alpha_i})}{p_i^{\alpha_i}}\geq
			ns \sqrt[s]{\frac{\prod \limits_{i=1}^s f(p_i^{\alpha_i})}{\prod \limits_{i=1}^s p_i^{\alpha_i}}}
		\end{align*}
		Since $\prod \limits_{i=1}^s f(p_i^{\alpha_i})=f(n)$ and $\prod \limits_{i=1}^s p_i^{\alpha_i}=n$ , Then we have : 
		$$
		\Phi_f(n) \geq ns \sqrt[s]{ \frac{f(n)}{n}} 
		$$
	\end{proof}
	\begin{proposition}
		Let $n$ a positive integer such that  and let $f$ an arithmetic function additive , then  we have : 
		\begin{equation}
			\Phi_{f}(n)\leq  \frac{1}{2}nf(n).
		\end{equation}
	\end{proposition}
	\begin{proof}
		Let $n$ a positive integer such that $n = p_1^{\alpha _1}\ldots  p_s^{\alpha _s}$ and let $f$ an arithmetic function , then we have .
		$$
		p^{\alpha}\geq 2 \Longrightarrow \frac{1}{p^{\alpha}}\leq \frac{1}{2} \Longrightarrow \frac{f(p^{\alpha})}{p^{\alpha}}\leq \frac{f(p^{\alpha})}{2} \Longrightarrow 
		\sum \limits_{p^{\alpha}||n} \frac{f(p^{\alpha})}{p^{\alpha}}\leq \frac{1}{2}\sum \limits_{p^{\alpha}||n} f(p^{\alpha})
		$$
	\end{proof}
	\begin{proposition}
		Let $n$ a positive integer .\\
		For every multiplicative function $f$ , there are exist a function additive $g$ such that : $\Phi_f(n)=\Phi_g(n)$ , with :
		$$
		g(n)=\sum \limits_{p^{\alpha}||n} f(p^{\alpha})
		$$
	\end{proposition}
	\begin{proof}
		Let \(f\) be a multiplicative function, and define \(g(n) = \sum_{p^{\alpha} \parallel n} f(p^{\alpha})\). Then, for any prime power \(p^{\alpha}\), we have \(g(p^{\alpha}) = f(p^{\alpha})\), implying \(\Phi_f(n) = \Phi_g(n)\).
	\end{proof}
	
	\begin{proposition}
		Given two multiplicative arithmetic functions $f$ and $g$  ,then for every positive integer $n$   we have : 
		\begin{equation}
			f(n)=g(n)\Longleftrightarrow \Phi_{f}(n)=\Phi_g(n)
		\end{equation}
	\end{proposition}

	After we have seen that many fundamental properties of the \textit{\textbf{additive transform}} of an arithmetic function. We complete this article by changing our point of view slightly and demonstrate that \textit{\textbf{additive transform}} can also be studied in terms of the Dirichlet convolutions .
	
	\medskip
	
	Let $f$ and $g$ be arithmetic functions. Their {\it Dirichlet convolution} is
	$$
	(f\ast g)(n)=\sum_{\substack{a,b=1\\ab=n}}^n f(a)g(b)=\sum_{\substack{d|n}}^n f(d)g\left(\frac{n}{d} \right) .
	$$
	
	where the sum extends over all positive divisors $d$ of $n$ , or equivalently over all distinct pairs $(a, b)$ of positive integers whose product is $n$.\\
	In particular, we have $(f*g)(1)=f(1)g(1)$ ,$(f*g)(p)=f(1)g(p)+f(p)g(1)$ for any prime $p$ and for any power prime $p^m$ we have :
	\begin{equation}
		(f*g)(p^m)=\sum \limits_{j=0}^m f(p^j)g(p^{m-j})
	\end{equation} 
	This product occurs naturally in the study of Dirichlet series such as the Riemann zeta function. It describes the multiplication of two Dirichlet series in terms of their coefficients: 
	\begin{equation}\label{eq:5}
		\bigg(\sum \limits_{n\geq 1}\frac{\big(f*g\big)(n)}{n^s}\bigg)=\bigg(\sum \limits_{n\geq 1}\frac{f(n)}{n^s} \bigg)
		\bigg( \sum \limits_{n\geq 1}\frac{g(n)}{n^s} \bigg)
	\end{equation}
	with Riemann zeta function or  is defined by : $$\zeta(s)= \sum \limits_{n\geq 1} \frac{1}{n^s}$$
	These functions are widely studied in the literature (see, e.g., \cite{book1, book2, book3}).\\

	\begin{proposition}\label{prop-2-5}
		Given two arithmetic functions $f$ and $g$, and lets $n$ and $m$ two positive integer such that $gcd(n,m)=1$ ,  if $f$ is multiplicative function ,Then we have :  
		\begin{equation}
			(f*\Phi_g)(nm)=(Id*f)(n).(f*\Phi_g)(m)+(Id*f)(m).(f*\Phi_g)(n)
		\end{equation}
	\end{proposition}
	\begin{proof}
		Lets $n$ and $m$ two positive integers such that $gcd(n,m)=1$, and lets $f$ and $g$  two arithmetics function  , if $f$ is multuplicative then we have :
		\begin{align*}
			(f*\Phi_g)(nm) &=\sum \limits_{d|nm} f\big(\frac{nm}{d}\big)\Phi_g(d)
			=\sum \limits_{\underset{d_2|m}{d_1|n}} f(\frac{nm}{d_1d_2})\Phi_g(d_1d_2)\\
			&=\sum \limits_{\underset{d_2|m}{d_1|n}} f(\frac{n}{d_1}) f(\frac{m}{d_2}) \bigg(d_1\Phi_g(d_2)+d_2\Phi_g(d_1)\bigg)
			\\
			&=\sum \limits_{\underset{d_2|m}{d_1|n}}\bigg(  d_1f(\frac{n}{d_1}) f(\frac{m}{d_2})\Phi_g(d_2)
			+d_2f(\frac{m}{d_2})f(\frac{n}{d_1})\Phi_g(d_1)
			\bigg)
			\\
			& =\bigg(\sum \limits_{d_1|n} d_1f(\frac{n}{d_1})\bigg)\bigg(\sum \limits_{d_2|m} f(\frac{m}{d_2})\Phi_g(d_2)\bigg)
			+
			\bigg(\sum \limits_{d_2|m} d_2f(\frac{m}{d_2})\bigg)\bigg(\sum \limits_{d_1|n} f(\frac{n}{d_1})\Phi_g(d_1)\bigg)
			\\
			&=\big(Id*f\big)(n).\big(f*\Phi_g\big)(m)+\big(Id*f\big)(m).\big(f*\Phi_g\big)(n)
		\end{align*}
	\end{proof}
	
	\begin{corollary}\label{cor-2-1}
		Let $n$ a positive integer   , for every multiplicative function we have : 
		\begin{equation}
			\big(f\ast\Phi_{g}\big)(n)
			=(Id*f)(n)
			\sum \limits_{p^{\alpha}||n} \frac{\big(f\ast\Phi_{g}\big)(p^{\alpha})}{(Id*f)(p^{\alpha})}
		\end{equation}
	\end{corollary}
	\begin{proof}
		Let $n$ a positive integer such that $n = p_1^{\alpha _1}\ldots  p_s^{\alpha _s}$ and lets $f$ and $g$ two arithmetics functions , if $f$ is multiplicative function Then by using the proposition $(\ref{prop-2-5})$ we have :
		\begin{align*}
			\big(f*\Phi_{g}\big)(n) &=\big(f*\Phi_{g}\big)(p_1^{\alpha _1}\ldots  p_s^{\alpha _s})
			\\
			&=
			\big(Id*f\big)(p_2^{\alpha _2}\ldots  p_s^{\alpha _s}).\big(f*\Phi_{g}\big)(p^{\alpha_1}_1)
			+
			\big(Id*f\big)(p^{\alpha_1}_1).\big(f*\Phi_{g}\big)(p_2^{\alpha _2}\ldots  p_s^{\alpha _s})
			\\
			& =\big(Id*f\big)(n).
			\frac{\big(f*\Phi_{g}\big)(p^{\alpha_1}_1)}{\big(Id*f\big)(p_1^{\alpha _1})}
			+
			\big(Id*f\big)(p^{\alpha_1}_1).
			\bigg[
			\big(Id*f\big)(p_3^{\alpha _3}\ldots  p_s^{\alpha _s}).\big(f*\Phi_{g}\big)(p^{\alpha_2}_2)
			+\\
			&+
			\big(Id*f\big)(p^{\alpha_2}_2).\big(f*\Phi_{g}\big)(p_3^{\alpha _3}\ldots  p_s^{\alpha _s})
			\bigg]
			\\
			&=\big(Id*f\big)(n)\frac{\big(f*\Phi_{g}\big)(p^{\alpha_1}_1)}{\big(Id*f\big)(p_1^{\alpha _1})}
			+
			\big(Id*f\big)(n)\frac{\big(f*\Phi_{g}\big)(p^{\alpha_2}_2)}{\big(Id*f\big)(p_2^{\alpha _2})}
			+\\
			&+
			\big(Id*f\big)(p^{\alpha_1}_1)\big(Id*f\big)(p^{\alpha_2}_2)\big(f*\Phi_{g}\big)(p_3^{\alpha _3}\ldots  p_s^{\alpha _s})
			\\
			\vdots\\
			&=\big(Id*f\big)(n)\frac{\big(f*\Phi_{g}\big)(p^{\alpha_1}_1)}{\big(Id*f\big)(p_1^{\alpha _1})}
			+ \ldots
			+
			\big(Id*f\big)(n)\frac{\big(f*\Phi_{g}\big)(p^{\alpha_s}_s)}{\big(Id*f\big)(p_s^{\alpha _s})}
			\\
			&=\big(Id*f\big)(n)
			\bigg[
			\frac{\big(f*\Phi_{g}\big)(p^{\alpha_1}_1)}{\big(Id*f\big)(p_1^{\alpha _1})}+\ldots
			+ \frac{\big(f*\Phi_{g}\big)(p^{\alpha_s}_s)}{\big(Id*f\big)(p_s^{\alpha _s})}
			\bigg]\\
			&=\big(Id*f\big)(n)\sum \limits_{i=1}^s
			\frac{\big(f*\Phi_{g}\big)(p^{\alpha_i}_i)}{\big(Id*f\big)(p^{\alpha_i}_i)}
		\end{align*}
	\end{proof}
	
	\begin{lemma}\label{lem-2-2}
		Let $p$ be a number prime ,and let $f$ an arithmetic function , then we have :
		\begin{equation}
			\forall\;\alpha\;\in\mathbb{N}^*\;\;\;,\;\big(\mu*\Phi_f\big)(p^\alpha)=
			\left\lbrace
			\begin{array}{lll}
				f(p^\alpha)-f(p^{\alpha-1})  &  if  & \alpha>1  \\
				f(p)  &  if  & \alpha=1
				
			\end{array}
			\right.
		\end{equation}
	\end{lemma}
	\begin{proof}
		Let $p$ be a number prime ,and $\alpha\;\in\mathbb{N}^*$  , if $\alpha=1$ then :
		$$
		\big(\mu*\Phi_f\big)(p^1)=\mu(1)\Phi_f(p)+\mu(p)\Phi_f(1)=f(p)
		$$
		If $\alpha>1$ then we have :
		$$
		\big(\mu*\Phi_f\big)(p^\alpha)=\sum \limits_{j=0}^{\alpha-1} \mu(p^j)\Phi_f(p^{\alpha-j})=
		\sum \limits_{j=0}^{\alpha-1} \mu(p^j)f(p^{\alpha-j})=
		\mu(p^0)f(p^{\alpha})+\ldots+\mu(p^{\alpha-1})f(p)
		$$
		Since for every integer $m\geq 2$ we know that $\mu(p^m)=0$ , then : 
		$$
		\big(\mu*\Phi_f\big)(p^\alpha)=\mu(1)f(p^{\alpha})+\mu(p)f(p^{\alpha-1})=f(p^\alpha)-f(p^{\alpha-1})
		$$
	\end{proof}
	
	\begin{proposition}\label{prop-2-6}
		Let $n$ a positive integer such that $n = p_1^{\alpha _1}\ldots  p_s^{\alpha _s}$ and let $f$ an arithmetic function , we have :  
		\begin{equation}
			\big(\mu*\Phi_f\big)(n)=
			\left\lbrace
			\begin{array}{lll}
				\varphi(n)\sum \limits_{p^{\alpha}||n}
				\frac{f(p^{\alpha})-f(p^{\alpha-1})}{p^{\alpha}-p_i^{\alpha-1}}  &  if  & n \text{ is a perfect square}  \\
				\varphi(n)\sum \limits_{p^{\alpha}||n}
				\frac{f(p)}{p-1}  &  if  & n \text{ is a squarefree}	
			\end{array}
			\right.
		\end{equation}
	\end{proposition}
	\begin{proof}
		Let $n$ a positive integer ,Since $\mu*Id=\varphi$ , then substitiling $f=\mu$ into corollary $(\ref{cor-2-1})$ to found : 
		$$
		\big(\mu\ast\Phi_{f}\big)(n)
		=(Id*\mu)(n)
		\sum \limits_{p^{\alpha}||n} \frac{\big(\mu\ast\Phi_{f}\big)(p^{\alpha})}{(Id*\mu)(p^{\alpha})}
		=\varphi(n)
		\sum \limits_{p^{\alpha}||n} \frac{\big(\mu\ast\Phi_{f}\big)(p^{\alpha})}{\varphi(p^{\alpha})}
		$$
		Since $\varphi(p^{\alpha})=p^{\alpha}-p^{\alpha-1}$ and by usine Lemma $(\ref{lem-2-2})$ we have : 
		$$
		\big(\mu*\Phi_f\big)(n)=
		\left\lbrace
		\begin{array}{lll}
			\varphi(n)\sum \limits_{p^{\alpha}||n}
			\frac{f(p^{\alpha})-f(p^{\alpha-1})}{p^{\alpha}-p^{\alpha-1}}  &  if  & n \text{ is a perfect square}  \\
			\varphi(n)\sum \limits_{p^{\alpha}||n}
			\frac{f(p)}{p-1}  &  if  & n \text{ is a squarefree}	
		\end{array}
		\right.
		$$
	\end{proof}
	\begin{remark}
		if $f$ is completely additive then we have : 
		\begin{equation}
			\big(\mu*\Phi_f\big)(n)=\sum \limits_{p^{\alpha}||n}\frac{f(p)}{p^{\alpha}-p_i^{\alpha-1}}
		\end{equation}
	\end{remark}
	\begin{proposition}\label{prop-2-7}
		Let $n$ a positive integer   ,if $f$ is an arithmetic function additive then  we have : 
		\begin{equation}
			\Phi_{Id.f}(n)=nf(n)
		\end{equation}
	\end{proposition}	
	\begin{proof}
		Let $n$ a positive integer   ,and $f$ is an arithmetic function additive, then by using the theorem $(\ref{the-2-1})$ we have :
		$$
		\Phi_{Id.f}(n)=n\sum \limits_{p^{\alpha}||n}\frac{Id\big(p^{\alpha}\big)f\big(p^{\alpha}\big)}{p^{\alpha}}
		=n\sum\limits_{p^{\alpha}||n}\frac{p^{\alpha}f\big(p^{\alpha}\big)}{p^{\alpha}}
		=n\sum\limits_{p^{\alpha}||n} f\big(p^{\alpha}\big)
		=nf(n)
		$$
	\end{proof}
	\begin{theorem}
		Let $n$ a positive integer   , for every arithmetic function additive $f$  we have : 
		\begin{equation}
			\big(\mu*f.Id\big)(n)
			=\varphi(n)f(n)+\varphi(n)
			\sum \limits_{p^{\alpha}||n} \frac{f(p^{\alpha})-f(p^{\alpha-1})}{p-1}
		\end{equation}
	\end{theorem}
	
	\begin{proof}
		Let $n$ a positive integer   , and $f$ an arithmetic function additive, by using proposition $(\ref{prop-2-6})$ we have :
		\begin{align*}
			\big(\mu\ast\Phi_{Id.f}\big)(n) &=\varphi(n)\sum \limits_{p^{\alpha}||n}
			\frac{p^{\alpha}f(p^{\alpha})-p^{\alpha-1}f(p^{\alpha-1})}{p^{\alpha}-p^{\alpha-1}}
			\\
			&=\varphi(n)\sum \limits_{p^{\alpha}||n}
			\frac{p^{\alpha}f(p^{\alpha})-p^{\alpha-1}f(p^{\alpha})+p^{\alpha-1}f(p^{\alpha})-p^{\alpha-1}f(p^{\alpha-1})}{p^{\alpha}-p^{\alpha-1}}
			\\
			&=\varphi(n)\sum \limits_{p^{\alpha}||n}
			\frac{f(p^{\alpha})(p^{\alpha}-p^{\alpha-1})}{p^{\alpha}-p^{\alpha-1}}+
			\frac{p^{\alpha-1}(f(p^{\alpha})-f(p^{\alpha-1}))}{p^{\alpha-1}(p-1)}
			\\
			&=\varphi(n)\sum \limits_{p^{\alpha}||n} f(p^{\alpha})+\varphi(n)
			\sum \limits_{p^{\alpha}||n} \frac{f(p^{\alpha})-f(p^{\alpha-1})}{p-1}
			\\
			&=\varphi(n)f(n)+\varphi(n)
			\sum \limits_{p^{\alpha}||n} \frac{f(p^{\alpha})-f(p^{\alpha-1})}{p-1}
		\end{align*}
		Since $\big(\mu\ast\Phi_{Id.f}\big)(n)=nf(n)$ by the proposition $(\ref{prop-2-7})$ then we have :
		$$
		\big(\mu*f.Id\big)(n)
		=\varphi(n)f(n)+\varphi(n)
		\sum \limits_{p^{\alpha}||n} \frac{f(p^{\alpha})-f(p^{\alpha-1})}{p-1}
		$$
	\end{proof}
	If f is completely additive then $f(p^{\alpha})-f(p^{\alpha-1})=\alpha f(p)-(\alpha-1)f(p)=f(p)$ , so we have this corollary :
	\begin{corollary}\label{cor-2-2}
		Let $n$ a positive integer not null , If $f$ is completely additive then we have : 
		\begin{equation}
			\big(\mu*Id.f\big)(n)=\varphi(n)f(n)+\varphi(n)
			\sum \limits_{p^{\alpha}||n} \frac{f(p)}{p-1}		
		\end{equation}
	\end{corollary}

\end{document}